\documentclass[12pt,reqno]{amsart}

\usepackage[T1]{fontenc}
\usepackage[utf8]{inputenc}
\usepackage{lmodern}
\usepackage{microtype}
\usepackage{amsmath,amssymb,amsthm,mathtools}
\usepackage{aliascnt}
\usepackage{enumitem}
\usepackage{xcolor}
\usepackage{hyperref}
\usepackage[nameinlink,capitalise,noabbrev]{cleveref}
\usepackage[margin=1.25in]{geometry}
\usepackage{booktabs}
\usepackage{tikz}
\usetikzlibrary{arrows.meta,decorations.pathreplacing}

\makeatletter
\def\thmhead@plain#1#2#3{%
  \thmname{#1}\thmnumber{\@ifnotempty{#1}{ }\@upn{#2}}%
  \thmnote{ \the\thm@notefont(#3)}}
\let\thmhead\thmhead@plain
\makeatother

\hypersetup{
  colorlinks=true,
  linkcolor=blue!50!black,
  citecolor=blue!50!black,
  urlcolor=blue!50!black,
  pdftitle={Counterexamples for lacunary dilates via dyadic spike blocks},
  pdfauthor={Boon Suan Ho}
}

\newtheorem{theorem}{Theorem}[section]
\newaliascnt{proposition}{theorem}
\newtheorem{proposition}[proposition]{Proposition}
\aliascntresetthe{proposition}
\newaliascnt{lemma}{theorem}
\newtheorem{lemma}[lemma]{Lemma}
\aliascntresetthe{lemma}
\newaliascnt{corollary}{theorem}

\aliascntresetthe{corollary}
\newaliascnt{definition}{theorem}
\newtheorem{definition}[definition]{Definition}
\aliascntresetthe{definition}
\theoremstyle{remark}
\newaliascnt{remark}{theorem}
\newtheorem{remark}[remark]{Remark}
\aliascntresetthe{remark}

\newcommand{\T}{\mathbb T}
\newcommand{\Z}{\mathbb Z}
\newcommand{\R}{\mathbb R}

\newcommand{\1}{\mathbf 1}

\newcommand{\norm}[1]{\left\lVert #1\right\rVert}

\newcommand{\val}{\operatorname{\nu}_2}

\numberwithin{equation}{section}

\crefname{theorem}{Theorem}{Theorems}
\crefname{proposition}{Proposition}{Propositions}
\crefname{lemma}{Lemma}{Lemmas}
\crefname{corollary}{Corollary}{Corollaries}
\crefname{definition}{Definition}{Definitions}
\crefname{remark}{Remark}{Remarks}

\begin{document}

\title[Dyadic spike blocks for lacunary dilates]
{Counterexamples for lacunary dilates\\ via dyadic spike blocks}
\author{Boon Suan Ho}
\address{Department of Mathematics, National University of Singapore}
\email{\href{mailto:hbs@u.nus.edu}{hbs@u.nus.edu}}
\date{}



\begin{abstract}
We construct counterexamples, along dyadic lacunary sequences, to two problems
of Erd\H{o}s on the almost-everywhere behavior of sums
\(\sum_{j\le N}f(n_jx)\) on the circle.  The main device is a dyadic spike
block: rare positive spikes, each counted many times by a run of consecutive
dilates, make the lacunary averages large, while a deterministic lower bound on
every block prevents the rest of the series from cancelling them.

We obtain a mean-zero \(f\in\bigcap_{1\le q<\infty}L^q(\T)\) and a sequence
\(n_j=2^{m_j}\) with \(n_{j+1}/n_j\ge2\) such that
\[
        \norm{f-S_Nf}_2\ll (\log\log N)^{-1/2},
        \qquad
        \limsup_{N\to\infty}
        \frac1N\sum_{j\le N}f(n_jx)=+\infty
\]
for almost every \(x\).  This answers negatively a question of Erd\H{o}s on
whether the condition \(\norm{f-S_Nf}_2\ll(\log\log\log N)^{-C}\) suffices for
almost-everywhere convergence, and shows that Matsuyama's condition
\(\norm{f-S_Nf}_2\ll(\log\log N)^{-c}\), \(c>1/2\), cannot be relaxed to
\(c=1/2\).  A second choice of parameters gives, for every \(2\le p<\infty\),
functions \(f\in L^p(\T)\) with
\[
        \limsup_{N\to\infty}
        \frac{\sum_{j\le N}f(n_jx)}
             {N(\log N)^{1/p-\varepsilon}}
        =+\infty
        \qquad(\varepsilon>0)
\]
almost everywhere.  This is sharp up to \(\varepsilon\), and the case \(p=2\)
answers negatively a second question of Erd\H{o}s.  Finally, a bounded variant
in the spirit of Marstrand's disproof of Khintchine's conjecture gives sets of
arbitrarily small measure whose lacunary visiting frequencies have limsup \(1\)
almost everywhere.
\end{abstract}

\maketitle

\section{Introduction}

Let $\T=\R/\Z$ with normalized Lebesgue measure.  For $f\in L^2(\T)$ write
\[
        S_N f(x)=\sum_{|m|\le N}\widehat f(m)e^{2\pi imx}
\]
for the symmetric Fourier partial sum.  An increasing integer sequence
$n_1<n_2<\cdots$ is \emph{lacunary} if $n_{j+1}\ge \rho n_j$ for some $\rho>1$ and all
$j$.  We will construct dyadic lacunary sequences $n_j=2^{m_j}$ in this paper, so that
$n_{j+1}/n_j\ge2$.

Erd\H{o}s asked in \cite{Erdos1964} whether a very weak Fourier-tail condition
of the form
\[
        \norm{f-S_Nf}_2\ll (\log\log\log N)^{-C}
\]
forces the lacunary averages $N^{-1}\sum_{j\le N}f(n_jx)$ to converge almost
everywhere for every lacunary sequence $(n_j)$.  Earlier positive results used stronger
conditions: Kac--Salem--Zygmund \cite{KacSalemZygmund} assumed logarithmic decay,
Erd\H{o}s \cite{Erdos1949} assumed double-logarithmic decay with exponent
$c>1$, and Matsuyama \cite{Matsuyama1966} reached
\[
        \norm{f-S_Nf}_2\ll (\log\log N)^{-c},\qquad c>1/2.
\]
Raikov's theorem covers the special case of exact geometric dilates $n_k=a^k$;
see \cite{Raikov1936}.  For the full sequence $n_j=j$ (Khintchine's problem,
discussed before \cref{thm:bounded}), Koksma \cite{Koksma1951} proved
almost-everywhere convergence under the condition
$\sum_m|\widehat f(m)|^2(\log\log m)^3<\infty$, and later \cite{Koksma1953}
under the weaker condition
\[
        \sum_m|\widehat f(m)|^2\sigma_{-1}(m)<\infty,
        \qquad
        \sigma_{-1}(m)=\sum_{d\mid m}\frac1d\ll\log\log m.
\]
Bourgain
remarked that Koksma's double-logarithmic condition is essentially best
possible \cite[p.~89]{Bourgain1988}, and Berkes and Weber
\cite{BerkesWeber2014} proved that the condition with $\sigma_{-1}$ is optimal in
a precise sense.  For classical and modern background on lacunary
series and systems of dilated functions, see
\cite{Gaposhkin1966,AistleitnerBerkesTichy2024,AistleitnerBerkesSeip2015,AistleitnerBerkesSeipWeber2015}.

Erd\H{o}s also asked about the largest possible almost-sure order of partial sums
\[
        \sum_{j\le N} f(n_jx)
\]
for $f\in L^2(\T)$ and lacunary $(n_j)$.  His examples gave lower bounds
of order
\[
        N(\log\log N)^{1/2-\varepsilon},
\]
while his general upper bound had size $N(\log N)^{1/2+\varepsilon}$
\cite{Erdos1949,Erdos1964}.  In particular he asked whether
\[
        \sum_{j\le N} f(n_jx)=o\bigl(N\sqrt{\log\log N}\bigr)
\]
must hold almost everywhere.

The present paper gives negative answers to both questions.  (Both also appear
in Bloom's online collection of Erd\H{o}s problems
\cite{Bloom995,Bloom996}.)

\subsection{Main results}

The central result is the endpoint Fourier-tail counterexample.

\begin{theorem}[Endpoint Fourier-tail counterexample]\label{thm:endpoint}
There exist a real-valued mean-zero function $f\in\bigcap_{1\le p<\infty}L^p(\T)$ and a lacunary
integer sequence $(n_j)$ with $n_{j+1}/n_j\ge2$ such that
\begin{equation}\label{eq:endpoint-tail}
        \norm{f-S_Nf}_2\ll (\log\log N)^{-1/2}
\end{equation}
for all sufficiently large $N$, while
\begin{equation}\label{eq:endpoint-divergence}
        \limsup_{N\to\infty}\frac{1}{N}\sum_{j\le N}f(n_jx)=+\infty
\end{equation}
for almost every $x\in\T$.
\end{theorem}

Since $(\log\log N)^{-1/2}\le(\log\log\log N)^{-C}$ for every fixed $C>0$ and
all large $N$, \cref{thm:endpoint} answers the first question of Erd\H{o}s
negatively; it also shows that the range $c>1/2$ in Matsuyama's theorem cannot
be extended to $c=1/2$.

The same construction, with different parameters, gives large partial sums in
every finite $L^p$, $p\ge2$.

\begin{theorem}[Large partial sums in $L^p$]\label{thm:large-partial-sums}
Let $2\le p<\infty$.  There exist a real-valued mean-zero function
$f\in L^p(\T)$ and a lacunary integer sequence $(n_j)$ with
$n_{j+1}/n_j\ge2$ such that, for almost every $x\in\T$,
\begin{equation}\label{eq:large-partial-sums}
        \limsup_{N\to\infty}
        \frac{\sum_{j\le N} f(n_jx)}
        {N(\log N)^{1/p-\varepsilon}}
        =+\infty
        \qquad\text{for every }\varepsilon>0.
\end{equation}
In particular, taking $p=2$ gives an $L^2$ example for which
$\sum_{j\le N} f(n_jx)$ is not $o(N\sqrt{\log\log N})$ almost everywhere.
\end{theorem}

This answers the second question of Erd\H{o}s negatively.  The exponent $1/p$
cannot be improved: \cref{rem:lp-upper} shows that
$\sum_{j\le N}f(n_jx)=O(N(\log N)^{1/p+\varepsilon})$ almost everywhere for
every $f\in L^p(\T)$ and every increasing integer sequence $(n_j)$.

Finally we include a bounded companion construction.  Its setting is that of
Khintchine's problem \cite{Khintchine1923}, which asks whether
$N^{-1}\sum_{n\le N}\1_E(nx)\to|E|$ almost everywhere for every measurable
$E\subset\T$.  Marstrand \cite{Marstrand1970} showed that the answer is
negative in a strong form: for every $\varepsilon>0$ there is an open set
$B\subset\T$ with $|B|<\varepsilon$ such that
$\limsup_{N\to\infty}N^{-1}\sum_{n\le N}\1_B(nx)=1$ for every $x$.  Bourgain
\cite{Bourgain1988} later derived the failure of almost-everywhere convergence
for some bounded measurable function from his entropy criterion, using part
of Marstrand's construction.
\Cref{thm:bounded} is a lacunary analogue of Marstrand's theorem.  For lacunary
sequences, stronger qualitative small-set statements are
known from the sweeping-out literature \cite{ABJLRW1996,MondalRoyWierdl2023};
the point here is to show how the same construction adapts to the bounded
setting.

\begin{theorem}[Bounded companion construction]\label{thm:bounded}
For every $0<\varepsilon<1$ there exist a measurable set $E\subset\T$ and a
lacunary integer sequence $(n_j)$ with $n_{j+1}/n_j\ge2$ such that
$|E|<\varepsilon$ and, for almost every $x\in\T$,
\[
        \limsup_{N\to\infty}\frac{1}{N}\sum_{j\le N}\1_E(n_jx)=1.
\]
Consequently the bounded mean-zero function $\1_E-|E|$ has lacunary
averages which fail to converge almost everywhere along this sequence.
\end{theorem}

\subsection{Overview of the construction}\label{subsec:roadmap}

We describe the construction behind \cref{thm:endpoint}; the other two theorems
use the same scheme with different parameters.  The function and the sequence
have the form
\[
        f=\sum_{k\ge1}F_k,
        \qquad
        n_j=2^{m_j},
\]
and are built in stages: stage $k$ adds one block $F_k$ to $f$ and a finite
batch of new exponents $m_j$.  We regard $\T$ as a probability space, so that
the binary digits of a point $x\in\T$ are independent fair bits, and functions
of disjoint sets of digits are independent random variables.

\medskip\noindent\emph{Spikes and blocks.}
The basic building block is the dyadic spike $\phi_d$, which is defined so that
it has mean zero and $L^2$ norm one.  It is positive of size about $2^{d/2}$ on
an interval of length $2^{-d}$, and negative of size about $2^{-d/2}$ everywhere
else: a rare large positive value, and a tiny negative value the rest of the
time.  The $k$th block is a normalized sum of dilated spikes,
\[
        F_k(x)=\sqrt{\frac{\lambda_k}{L_k}}
        \sum_{q=1}^{L_k}\phi_{d_k}(2^{U_k+qD_k}x),
\]
where the spacing $D_k=d_k+2$ makes the $L_k$ summands depend on disjoint
sets of binary digits, hence independent and orthogonal, and the shift $U_k$
moves $F_k$ past the digits and Fourier valuations used by earlier blocks, so
that different blocks are also independent and orthogonal.  Two parameters
govern stage $k$: the \emph{cost} $\lambda_k=\norm{F_k}_2^2$ and the
\emph{target height} $B_k$.  The depth $d_k$ is chosen with
$2^{d_k}\asymp B_k^2L_k/\lambda_k$, which makes a single positive spike, after
normalization, of height of order $B_k$.  With this choice the negative values of
all $L_k$ spikes together add up to very little:
\[
        F_k(x)\ge -\frac{\lambda_k}{4B_k}
        \qquad(x\in\T).
\]
This deterministic lower bound is what makes the construction work.  Since
$\sum_k\lambda_k/B_k<\infty$, every partial sum of $\sum_kF_k$ is bounded below
by a fixed constant $-\mu$, so a positive signal produced by one block cannot be
cancelled by the others.

\medskip\noindent\emph{Trials and local amplification.}
A \emph{trial} is a run of exponents with the same spacing as the block,
\[
        M+D_k,\ M+2D_k,\ \ldots,\ M+\ell D_k,
        \qquad \ell\le L_k/8.
\]
Since $2^{U_k+qD_k}\cdot2^{M+rD_k}=2^{U_k+M+(q+r)D_k}$, summing $F_k$ over the
trial regroups the spikes according to $h=q+r$:
\[
        \sum_{r=1}^{\ell}F_k(2^{M+rD_k}x)
        =\sqrt{\frac{\lambda_k}{L_k}}
        \sum_h w_h\,\phi_{d_k}(2^{U_k+M+hD_k}x),
\]
where $w_h$ is the number of pairs $(q,r)$ with $q+r=h$.  For the central
indices $\ell<h\le L_k+1$ one has $w_h=\ell$: each central spike is seen by all
$\ell$ terms of the trial.  So if at least one central spike fires, the trial sum is
about $\ell$ times the spike height, and in fact at least $2B_k\ell$ once the
floor is taken into account.  There are about $L_k$ central spikes, each firing
with probability $2^{-d_k}\asymp\lambda_k/(B_k^2L_k)$, so a trial succeeds with
probability of order $\lambda_k/B_k^2$ (\cref{lem:local-amplification}).

\medskip\noindent\emph{From one trial to the whole average.}
A trial controls only its own $\ell$ terms, while the average
$N^{-1}\sum_{j\le N}f(n_jx)$ also contains every earlier term.  We therefore make
each trial long compared with everything before it: if $P$ exponents have been
selected when a trial begins, the trial has length $\ell\approx20P$.  At the
endpoint $N=P+\ell$ of a successful trial, the block $F_k$ contributes at least
$2B_k\ell$ along the trial.  Everything else, namely the $P$ earlier terms and
the blocks $F_i$, $i\ne k$, evaluated along the trial, is at least $-\mu$ per
term by the lower floor.  Hence
\[
        \frac1N\sum_{j\le N}f(n_jx)
        \ge\frac{2B_k\ell}{N}-\mu
        \ge B_k-\mu
\]
(\cref{prop:master}).  The price of this domination is that trial lengths grow
geometrically: each trial multiplies $1+P$ by a factor between $21$ and $22$.  The block length $L_k$ is chosen
only after the trial lengths of stage $k$ are known, as eight times the longest
of them.

\medskip\noindent\emph{Stages.}
Stage $k$ runs $T_k$ trials, placed so far apart in the binary expansion that
all good-trial events, within and across stages, are independent.  Hence stage
$k$ succeeds, meaning that at least one of its trials succeeds, with
probability at least $1-\exp(-cT_k\lambda_k/B_k^2)$, independently of the other
stages.  If $B_k\to\infty$ and $\sum_kT_k\lambda_k/B_k^2=\infty$, the second
Borel--Cantelli lemma gives infinitely many successful stages almost surely,
and so the lacunary averages have infinite limsup almost everywhere.

\medskip\noindent\emph{Where the exponent $1/2$ comes from.}
It remains to choose the parameters while keeping track of the Fourier tail.  We
run $T_k=\lceil1/\lambda_k\rceil$ trials, so that stage $k$ succeeds with
probability $\gg B_k^{-2}$.  We also take $\lambda_k$ so small that $T_k$
dominates the logarithms of all quantities from earlier stages.  By the
geometric growth of trial lengths, the number of exponents selected by the end
of stage $k$, and with it the block length $L_k$, is then
$\exp(O(1/\lambda_k))$.  The block $F_k$ involves dilations up to about
$2^{L_kD_k}$, so its Fourier transform is essentially supported below a
threshold $Q_k$ with
\[
        \log\log Q_k\asymp\log L_k\asymp\frac1{\lambda_k}.
\]
Now let $Q_k\le N<Q_{k+1}$.  Since the blocks have Fourier support in disjoint
valuation bands, $\norm{f-S_Nf}_2^2$ is the sum of the squared tails of the
individual blocks.  The current block contributes
$\ll\lambda_kQ_k/N\le\lambda_k\log\log Q_k/\log\log N\ll1/\log\log N$; the
earlier blocks contribute even less, because we also make the thresholds $Q_i$
grow rapidly; and the later blocks contribute at most
$\sum_{i>k}\lambda_i\ll\lambda_{k+1}\ll1/\log\log Q_{k+1}\le1/\log\log N$.  This
gives the endpoint bound $\norm{f-S_Nf}_2\ll(\log\log N)^{-1/2}$.  In other
words, just below $Q_{k+1}$ the squared tail is about
$\lambda_{k+1}\asymp1/\log\log Q_{k+1}$, and this balance is what produces the
exponent $1/2$.  The heights
can be any sequence with $B_k\to\infty$ and $\sum_kB_k^{-2}=\infty$, such as
$B_k\asymp\sqrt{\log k}$.  By Matsuyama's theorem no construction can reach an
exponent beyond $1/2$; \cref{rem:endpoint-limit} explains where the present one
meets this barrier.

\medskip\noindent\emph{Large partial sums in $L^p$.}
For \cref{thm:large-partial-sums} the Fourier tail plays no role, and the budget
is measured in $L^p$.  The spikes dominate the $L^p$ norm, and
$\norm{F_k}_p^p\ll\lambda_kB_k^{p-2}$, so we take $\lambda_k=a_kB_k^{-(p-2)}$ with
$\sum_ka_k<\infty$.  Since the Fourier tail no longer limits the number of
trials, we can afford $T_k\asymp(B_k^2/\lambda_k)\log k$ of them, which makes
stage failures summable, so that every large stage succeeds.  The geometric
growth of trial lengths gives $\log N\lesssim T_k$ at the stage-$k$ endpoints,
where $T_k\asymp B_k^p$ up to factors that do not depend on $B_k$.  A successful
trial therefore produces averages of size $B_k\gtrsim(\log N)^{1/p}$.  The
other factors are absorbed by taking $B_k$ very large, at the cost of an
$\varepsilon$ in the exponent.

\medskip\noindent\emph{The bounded construction.}
For \cref{thm:bounded} the spike block is replaced by its support: a union $E_k$
of $L_k$ dilated dyadic intervals, of small total measure.  A central hit now
places every one of the $\ell$ points of a trial in $E_k$.  Taking trial lengths
$\ell\approx(k+1)P$ in place of $20P$, the proportion of the first $N$ points
lying in $E$ is then at least $1-1/(k+2)$ at the trial endpoint.

\medskip\noindent\emph{Relation to earlier constructions.}
This scheme of stages and independent trials has two antecedents.  The first is the
construction in Erd\H{o}s's paper \cite{Erdos1949}.  There $f$ is a sum of
normalized Rademacher blocks, and the
exponents are chosen in many well-separated intervals.  On one interval the
current block contributes a long Rademacher sum multiplied by the length of the
interval; the different intervals are independent, and the other blocks are
controlled by $L^2$ estimates.  We keep the stages, the independent trials, and
the separate treatment of the other blocks, but replace the Rademacher block by
a sparse spike block.  A successful trial is then not a large Gaussian
fluctuation of many Rademacher functions, but a single rare spike counted $\ell$
times.  This keeps the block cheap in $L^p$ and in its Fourier tail, and the
lower floor replaces $L^2$ control of the other blocks by a pointwise bound.

The second, and closer, antecedent is Marstrand's disproof \cite{Marstrand1970}
of Khintchine's conjecture \cite{Khintchine1923}, which has the same three-level
structure.  For given $q$, Marstrand constructs a set $G^*$ and a much smaller
set $H^*$, both finite unions of short intervals with rational endpoints, such that
for every $x\in G^*$ there is $R\ge1$ such that a whole run of consecutive
dilates $\{a_rx\}$, $R\le r\le Rq$, lies in $H^*$
\cite[Lemma~3.4]{Marstrand1970}.  He then
takes independent dilated copies of $G^*$ to make this event almost certain
\cite[Lemma~3.5]{Marstrand1970}, and finally lets $q\to\infty$
\cite[Theorem~3.2]{Marstrand1970}; since the run $R\le r\le Rq$ is long
compared with the $r<R$ earlier terms, the averages of the indicator of the
small set come close to $1$.  These three steps correspond to our local
amplification lemma, to the independent trials within a stage, and to the
stages, in which a successful trial dominates everything before it.  For
$a_r=r$, the input to Marstrand's Lemma~3.4 comes from the divisibility
properties of integers with small prime factors
\cite[Sections~4--5]{Marstrand1970}.  In the dyadic setting it is replaced by
the identity $2^{U_k+qD_k}\cdot2^{M+rD_k}=2^{U_k+M+(q+r)D_k}$, and the spike
blocks carry the scheme over to unbounded mean-zero functions, where the lower
floor and the quantitative $L^p$ and Fourier-tail bounds become the main
issues.

\medskip\noindent\emph{Organization.}
Section~\ref{sec:spikes} introduces the dyadic spike and records its
distribution, independence, Fourier support, and Fourier tail.
Section~\ref{sec:local} builds a single block and proves the local amplification
lemma for a single trial.  Section~\ref{sec:master} assembles blocks and trials
into the general construction and proves the main estimate,
\cref{prop:master}.  Sections~\ref{sec:endpoint} and~\ref{sec:large-sums} choose
the parameters and prove \cref{thm:endpoint,thm:large-partial-sums},
respectively.  Section~\ref{sec:bounded} proves \cref{thm:bounded}, and
Section~\ref{sec:questions} collects some open questions.

\section{Dyadic spikes}\label{sec:spikes}

We identify $\T$ with $[0,1)$ and regard it, with Lebesgue measure $\Pr$, as a
probability space; functions on $\T$ are then random variables.  We discard the
dyadic rationals.  This is a null set, and its complement is mapped into itself
by every dilation $x\mapsto2^vx$, so every $x$ that remains has a unique binary
expansion $x=\sum_{i\ge1}x_i2^{-i}$, and the digits of $2^vx$ are
$x_{v+1},x_{v+2},\dots$.  The digits $x_1,x_2,\dots$ are independent fair bits.

For $d\ge1$ we define the \emph{spike} by
\begin{equation}\label{eq:phi-def}
        \phi_d(x)\coloneqq\frac{\1_{[0,2^{-d})}(x)-2^{-d}}{\sqrt{2^{-d}(1-2^{-d})}}
        =(h_d+g_d)\1_{[0,2^{-d})}(x)-g_d,
\end{equation}
where $h_d\coloneqq\sqrt{2^d-1}$ and $g_d\coloneqq1/\sqrt{2^d-1}$.
Then $\int_\T\phi_d=0$ and $\norm{\phi_d}_2=1$.

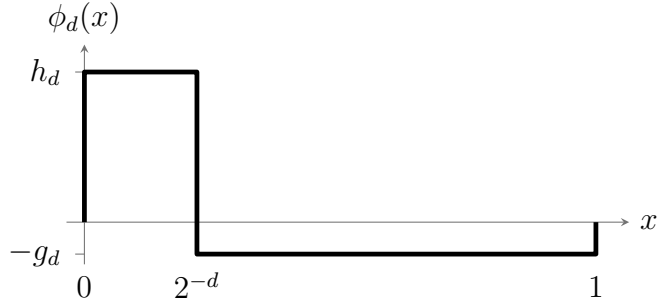
\begin{figure}[ht]
\centering
\begin{tikzpicture}[
  x=6.8cm,
  y=0.85cm,
  >=stealth,
  axis/.style={->,line width=0.4pt,black!55},
  tick/.style={line width=0.4pt,black!55},
  graph/.style={black,line width=1.8pt,line cap=butt,line join=round}
]
  \def\cutoff{0.22}
  \def\spikeheight{2.35}
  \def\floorheight{0.50}
  \def\xlabeldrop{-0.68}

  \draw[axis] (-0.035,0) -- (1.065,0)
    node[anchor=west,black] {$x$};
  \draw[axis] (0,-0.66) -- (0,2.78)
    node[anchor=south,black] {$\phi_d(x)$};

  \draw[tick] (-0.018,\spikeheight) -- (0,\spikeheight);
  \draw[tick] (-0.018,-\floorheight) -- (0,-\floorheight);

  \draw[graph]
    (0,0)
    -- (0,\spikeheight)
    -- (\cutoff,\spikeheight)
    -- (\cutoff,-\floorheight)
    -- (1,-\floorheight)
    -- (1,0);

  \node[anchor=north] at (0,\xlabeldrop-.01) {$0$};
  \node[anchor=north] at (\cutoff,\xlabeldrop+.07) {$2^{-d}$};
  \node[anchor=north] at (1,\xlabeldrop-.01) {$1$};
  \node[anchor=east] at (-0.025,\spikeheight) {$h_d$};
  \node[anchor=east] at (-0.025,-\floorheight) {$-g_d$};
\end{tikzpicture}
\caption{The spike $\phi_d$.
Here $h_d\coloneqq\sqrt{2^d-1}$ and $g_d\coloneqq1/\sqrt{2^d-1}$.}
\end{figure}

For a nonzero integer $r$, let $\val(r)$ be its dyadic valuation, that is, the
largest $v\ge0$ with $2^v\mid r$.  For integers $0\le a\le b$ we write
$[a,b]=\{a,a+1,\dots,b\}$ and call the set of frequencies
\[
        \{r\in\Z\setminus\{0\}:\val(r)\in[a,b]\}
\]
the \emph{dyadic valuation band} $[a,b]$.  This is an arithmetic, not an
Archimedean, notion: a valuation band contains frequencies of arbitrarily large
absolute value (for instance all odd multiples of $2^a$).  So Fourier support
in the band $[a,b]$ forces $|r|\ge2^a$, but gives no upper bound on $|r|$.
Valuation bands are used mainly for orthogonality.  If $b<a'$, the bands $[a,b]$ and $[a',b']$ are disjoint, so
two functions whose nonzero Fourier coefficients lie in
disjoint valuation bands are orthogonal; since $S_N$ acts diagonally on Fourier
coefficients, the same is true after applying $S_N$ or $I-S_N$ to both.  Upper
bounds on the frequencies, which govern Fourier tails, are obtained separately
in \cref{lem:spike-tail}.

\begin{lemma}[Distribution, independence, and Fourier support]\label{lem:spike-basic}
Let $d\ge1$.
\begin{enumerate}[label=(\alph*)]
\item The random variable $\phi_d(x)$ takes the values $h_d$ and $-g_d$ with
probabilities $2^{-d}$ and $1-2^{-d}$ respectively.
\item If $0\le v_1<\cdots<v_s$ and $v_{i+1}-v_i\ge d$, then the random variables
$\phi_d(2^{v_i}x)$, $1\le i\le s$, are independent.
\item For $r\ne0$, $\widehat{\phi_d}(r)=0$ whenever $2^d\mid r$.  Hence, for
$v\ge0$, the nonzero Fourier coefficients of $\phi_d(2^v x)$ occur only at
frequencies in the dyadic valuation band $[v,v+d-1]$.
\end{enumerate}
\end{lemma}

\begin{proof}
For $v\ge0$, the value of $\phi_d(2^v x)$ is determined by whether the first
$d$ binary digits of $2^v x$ are all zero, that is, by whether
$x_{v+1}=\dots=x_{v+d}=0$.  This gives (a), and (b) follows because functions
of disjoint sets of digits are independent.
For (c), when $r\ne0$,
\[
        \widehat{\1_{[0,2^{-d})}}(r)
        =\frac{1-e^{-2\pi i r2^{-d}}}{2\pi i r},
\]
which vanishes if $2^d\mid r$.  The constant subtraction in \eqref{eq:phi-def}
only affects the zero Fourier coefficient.  Dilating by $2^v$ shifts all dyadic
valuations by $v$.
\end{proof}

\begin{lemma}[Fourier tail of one spike]\label{lem:spike-tail}
There is an absolute constant $C_1>0$ such that, for every $d\ge1$ and $R\ge1$,
\begin{equation}\label{eq:spike-tail-base}
        \sum_{|r|>R}|\widehat{\phi_d}(r)|^2
        \le C_1\min\left(1,\frac{2^d}{R}\right).
\end{equation}
Consequently, for every $v\ge0$ and $N\ge1$,
\begin{equation}\label{eq:spike-tail-dilate}
        \norm{(I-S_N)\phi_d(2^v\cdot)}_2^2
        \le C_1\min\left(1,\frac{2^{d+v}}{N}\right).
\end{equation}
\end{lemma}

\begin{proof}
Parseval gives the bound by $1$.  The total variation of $\phi_d$ is
$2/(2^{-d}(1-2^{-d}))^{1/2}\ll 2^{d/2}$, so for $r\ne0$,
$|\widehat{\phi_d}(r)|\ll 2^{d/2}/|r|$.  Summing $r^{-2}$ over $|r|>R$ gives
\eqref{eq:spike-tail-base}.  The dilated estimate follows by applying
\eqref{eq:spike-tail-base} with $R=N/2^v$ when $N\ge2^v$, and by using the
trivial Parseval bound otherwise.
\end{proof}

\section{Blocks and local trials}\label{sec:local}

In this section we fix a single block and a single trial, and show that one
central hit produces a contribution proportional to the trial length
(\cref{lem:local-amplification}).

Fix once and for all
\begin{equation}\label{eq:B0}
        B_0=100.
\end{equation}
A \emph{block parameter set} consists of
\[
        0<\lambda\le1,
        \qquad B\ge B_0,
        \qquad L\ge1,
        \qquad d\ge1,
        \qquad D\ge d+2,
        \qquad U\ge0,
\]
with
\begin{equation}\label{eq:d-choice-local}
        64\frac{B^2L}{\lambda}
        \le 2^d
        <128\frac{B^2L}{\lambda}.
\end{equation}
As in Section~\ref{subsec:roadmap}, $\lambda$ is the cost of the block, $B$ the
target height, $L$ the number of spikes, $d$ their depth, $D$ their spacing,
and $U$ a base shift.  The associated block is
\begin{equation}\label{eq:block-def}
        F(x)=\sqrt\frac{\lambda}{L}
        \sum_{q=1}^{L}\phi_d(2^{U+qD}x).
\end{equation}

\begin{lemma}[Norm, lower floor, and Fourier tail of one block]\label{lem:block-basic}
Every block \eqref{eq:block-def} satisfies
\begin{equation}\label{eq:block-norm}
        \int_\T F=0,
        \qquad \norm F_2^2=\lambda,
\end{equation}
\begin{equation}\label{eq:block-lower-floor}
        F(x)\ge -\frac{\lambda}{4B}
        \qquad (x\in\T),
\end{equation}
and, for every $N\ge1$,
\begin{equation}\label{eq:block-tail}
        \norm{(I-S_N)F}_2^2
        \le \frac{C_1\lambda}{L}\sum_{q=1}^{L}
        \min\left(1,\frac{2^{d+U+qD}}{N}\right),
\end{equation}
where $C_1$ is the constant of \cref{lem:spike-tail}.  If
\begin{equation}\label{eq:Q-local}
        Q=2^{U+LD+d+2},
\end{equation}
then for every $N\ge Q$,
\begin{equation}\label{eq:block-tail-above-Q}
        \norm{(I-S_N)F}_2^2\le C_1\frac{\lambda Q}{N}.
\end{equation}
\end{lemma}

\begin{proof}
In $F(x)$, the index-$q$ summand has nonzero Fourier coefficients only in
the dyadic valuation band
\[
        [U+qD,\,U+qD+d-1].
\]
Since $D\ge d+2$, these bands are pairwise disjoint.  Hence the summands are
orthogonal, have mean zero, and have $L^2$ norm one.  This proves
\eqref{eq:block-norm}.

Since $\phi_d\ge -g_d$ with $g_d=(2^d-1)^{-1/2}\le\sqrt2\,2^{-d/2}$,
\[
        F(x)\ge -\sqrt\frac{\lambda}{L}\,L g_d
        =-\sqrt{\lambda L}\,g_d.
\]
The choice \eqref{eq:d-choice-local} gives
$2^{-d/2}\le (1/8)\sqrt{\lambda/(B^2L)}$, and hence
\[
        \sqrt{\lambda L}\,g_d
        \le \frac{\sqrt2}{8}\,\frac{\lambda}{B}
        \le \frac{\lambda}{4B}.
\]
This proves \eqref{eq:block-lower-floor}.

Since $I-S_N$ preserves the disjointness of the valuation bands, the summands
remain orthogonal after applying $I-S_N$, and \eqref{eq:block-tail} follows from
\cref{lem:spike-tail}.  If $N\ge Q$, we bound every minimum in
\eqref{eq:block-tail} by its second term and use
\[
        \sum_{q=1}^{L}2^{d+U+qD}
        \le 2^{d+U+LD+1}\le Q
\]
and $L\ge1$; this proves \eqref{eq:block-tail-above-Q}.
\end{proof}

For the $L^p$ statements in \cref{sec:endpoint,sec:large-sums} we also need the
$L^p$ size of a block.

\begin{lemma}[$L^p$ size of one block]\label{lem:block-lp}
For every $2\le p<\infty$ there is a constant $C_p$ such that every block
\eqref{eq:block-def} satisfies
\begin{equation}\label{eq:block-lp-bound}
        \norm{F}_p^p\le C_p\lambda B^{p-2}.
\end{equation}
\end{lemma}

\begin{proof}
The summands
\[
        X_q(x)=\sqrt{\lambda/L}\,\phi_d(2^{U+qD}x),
        \qquad 1\le q\le L,
\]
are independent and mean zero.  For $p=2$ the estimate is immediate from
\eqref{eq:block-norm}.  Assume $p>2$.  Rosenthal's inequality
\cite{Rosenthal1970} gives
\[
        \norm{F}_p^p
        \le C_p\left(\Big(\sum_{q=1}^L\norm{X_q}_2^2\Big)^{p/2}
        +\sum_{q=1}^L\norm{X_q}_p^p\right).
\]
Since $\|X_q\|_2^2=(\lambda/L)\|\phi_d(2^{U+qD}\cdot)\|_2^2=\lambda/L$,
$0<\lambda\le1$, and $B\ge1$,
the first term is $C_p\lambda^{p/2}\le C_p\lambda B^{p-2}$.
Also
\[
        \norm{\phi_d}_p^p
        =2^{-d}h_d^p+(1-2^{-d})g_d^p
        \le C_p2^{d(p/2-1)}.
\]
Therefore, using \eqref{eq:d-choice-local},
\begin{align*}
        \sum_{q=1}^L\norm{X_q}_p^p
        &\le C_pL\left(\frac{\lambda}{L}\right)^{p/2}
             2^{d(p/2-1)}  \\
        &\le C_pL\left(\frac{\lambda}{L}\right)^{p/2}
             \left(\frac{B^2L}{\lambda}\right)^{p/2-1}
         = C_p\lambda B^{p-2}.
\end{align*}
This proves \eqref{eq:block-lp-bound}.
\end{proof}

We shall also use the following standard summability criterion.

\begin{lemma}[Independent $L^p$ summability]\label{lem:lp-series}
Let $2\le p<\infty$, and let $(G_k)$ be independent mean-zero functions on
$\T$.  If
\[
        \sum_k\norm{G_k}_2^2<\infty,
        \qquad
        \sum_k\norm{G_k}_p^p<\infty,
\]
then $\sum_kG_k$ converges in $L^p$.
\end{lemma}

\begin{proof}
For $p=2$ this is just the Hilbert-space Cauchy criterion.  Assume $p>2$.
For $m\le n$, Rosenthal's inequality \cite{Rosenthal1970} gives
\[
        \norm{\sum_{k=m}^nG_k}_p^p
        \le C_p\left(\Big(\sum_{k=m}^n\norm{G_k}_2^2\Big)^{p/2}
        +\sum_{k=m}^n\norm{G_k}_p^p\right).
\]
The right-hand side tends to zero as $m,n\to\infty$.  Hence the partial sums
are Cauchy in $L^p$.
\end{proof}

A trial is specified by an integer starting offset $M$ and a length $\ell$,
where $1\le\ell\le L/8$.  We assume
\begin{equation}\label{eq:local-nonnegative-exponents}
        M+D\ge0,
\end{equation}
so that every dilation appearing below is an integer endomorphism of $\T$.  The
trial uses the exponents
\begin{equation}\label{eq:trial-local}
        M+D,\, M+2D,\ldots, M+\ell D.
\end{equation}
Define
\begin{equation}\label{eq:Zh-def}
        Z_h(x)=\phi_d(2^{U+M+hD}x),
        \qquad 2\le h\le L+\ell.
\end{equation}
Then
\begin{equation}\label{eq:convolution-identity}
        \sum_{r=1}^{\ell}F(2^{M+rD}x)
        =\sqrt\frac{\lambda}{L}
        \sum_{h=2}^{L+\ell} w_h Z_h(x),
\end{equation}
where
\[
        w_h=\#\{(q,r):1\le q\le L,\,1\le r\le\ell,\,q+r=h\}.
\]
Always $0\le w_h\le\ell$, and in the central range
\begin{equation}\label{eq:central-weights}
        w_h=\ell,
        \qquad \ell+1\le h\le L+1.
\end{equation}

\begin{definition}[Good trial event]\label{def:local-good}
The good event $\mathcal E(M,\ell)$ is the event that at least one central variable
$Z_h$, $\ell+1\le h\le L+1$, equals $h_d$.
\end{definition}

\begin{lemma}[Local amplification]\label{lem:local-amplification}
There is an absolute constant $c_0>0$ such that, for every block parameter set,
every integer $M$ satisfying \eqref{eq:local-nonnegative-exponents}, and every
$1\le\ell\le L/8$,
\begin{equation}\label{eq:local-prob}
        \Pr(\mathcal E(M,\ell))\ge c_0\frac{\lambda}{B^2}.
\end{equation}
On $\mathcal E(M,\ell)$ one has
\begin{equation}\label{eq:local-signal}
        \sum_{r=1}^{\ell}F(2^{M+rD}x)
        \ge 2B\ell.
\end{equation}
\end{lemma}

\begin{proof}
The central variables are independent by \cref{lem:spike-basic}, because their
digit windows have length $d$ and are separated by $D\ge d+2$.  Let
$p=2^{-d}$.  By \eqref{eq:d-choice-local},
\begin{equation}\label{eq:p-bounds}
        \frac{\lambda}{128B^2L}<p\le \frac{\lambda}{64B^2L}.
\end{equation}
The number of central indices is $n_c=L-\ell+1\ge7L/8$.  Since
$n_cp\le \lambda/(64B^2)\le1/64$, we have
\[
        \Pr(\mathcal E(M,\ell))=1-(1-p)^{n_c}
        \ge \frac{1}{2}n_cp
        \ge c_0\frac{\lambda}{B^2}.
\]

Suppose $Z_{h_0}=h_d$ for some central $h_0$.  Using
$\sum_h w_h=L\ell$, $w_{h_0}=\ell$, and $Z_h\ge -g_d$ for all $h$,
\begin{align*}
\sum_{r=1}^{\ell}F(2^{M+rD}x)
&\ge \sqrt\frac{\lambda}{L}
     \left(\ell h_d-L\ell g_d\right)  \\
&=\ell\left(\sqrt\frac{\lambda}{L}h_d-
             \sqrt{\lambda L}\,g_d\right).
\end{align*}
The choice \eqref{eq:d-choice-local} implies
\[
        \sqrt\frac{\lambda}{L}\,h_d
        \ge \frac{1}{\sqrt2}\sqrt{\frac{\lambda2^d}{L}}
        \ge 4\sqrt2\,B,
\]
where we used $h_d=\sqrt{2^d-1}\ge 2^{d/2}/\sqrt2$.  As in the proof of
\cref{lem:block-basic}, also $\sqrt{\lambda L}\,g_d\le\lambda/(4B)\le B$.  Thus
the expression in parentheses is at least $2B$, proving
\eqref{eq:local-signal}.
\end{proof}

\section{The general construction}\label{sec:master}

This section assembles blocks and trials into a function and a lacunary
sequence.  The stage parameters $\lambda_k$, $B_k$, and $T_k$ are left free
here; they are chosen in \cref{sec:endpoint,sec:large-sums}.
\Cref{tab:notation} lists the quantities attached to stage $k$.

\begin{table}[ht]
\centering
\small
\caption{Quantities attached to stage $k$.}\label{tab:notation}
\begin{tabular}{@{}ll@{}}
\toprule
$\lambda_k,\ B_k,\ T_k$ & cost, target height, and number of trials of stage $k$ \\
$P_{k,t}$ & number of exponents selected before trial $t$ of stage $k$ \\
$\ell_{k,t},\ N_{k,t}$ & length and endpoint $P_{k,t}+\ell_{k,t}$ of trial $t$ \\
$N_k^*$ & number of exponents selected by the end of stage $k$ \\
$F_k$ & the block of stage $k$ \\
$L_k,\ d_k,\ D_k,\ U_k$ & number of spikes, depth, spacing, and base shift of $F_k$ \\
$\mathcal V_k,\ Q_k$ & valuation bands and Fourier threshold of $F_k$ \\
$M_{k,t},\ \mathcal I_{k,t}$ & offset and set of exponents of trial $t$ \\
$\mathcal E_{k,t},\ S_k$ & good event of trial $t$; at least one good trial in stage $k$ \\
$\Omega_k$ & upper bound for the binary digits used by good events of stages $1,\dots,k$ \\
\bottomrule
\end{tabular}
\end{table}

Fix
\begin{equation}\label{eq:eta}
        \eta=\frac{1}{20}.
\end{equation}
Each trial will be at least $\eta^{-1}$ times as long as the number of exponents
selected before it.

Suppose stages $1,\ldots,k-1$ have been constructed.  We record four quantities
from them: the number $P_{k,1}$ of exponents selected so far; the largest
selected exponent $m^{\rm last}_{k-1}$; the largest dyadic valuation
$V_{k-1}^{\max}$ used by the blocks $F_1,\dots,F_{k-1}$; and an upper bound
$\Omega_{k-1}$ for the binary digits on which the good-trial events of these
stages depend.  For
$k=1$ all four are $0$.  Stage $k$ begins with a choice of
\begin{equation}\label{eq:stage-free-parameters}
        0<\lambda_k\le1,
        \qquad B_k\ge B_0,
        \qquad T_k\ge1.
\end{equation}

\medskip\noindent\emph{Trial lengths and the block.}
The trial lengths are defined recursively by
\begin{equation}\label{eq:length-recursion}
        \ell_{k,t}=\left\lceil \eta^{-1}(P_{k,t}+1)\right\rceil,
        \qquad P_{k,t+1}=P_{k,t}+\ell_{k,t}
        \quad(1\le t\le T_k).
\end{equation}
Set
\begin{equation}\label{eq:Nstar-def}
        N_k^*=P_{k,T_k+1},
        \qquad
        L_k=8\ell_{k,T_k}.
\end{equation}
Since $\ell_{k,t}$ is increasing in $t$, the block length $L_k$ is eight times
the longest trial length in the stage.  Choose $d_k$ by
\begin{equation}\label{eq:dk-choice}
        64\frac{B_k^2L_k}{\lambda_k}
        \le 2^{d_k}
        <128\frac{B_k^2L_k}{\lambda_k},
\end{equation}
put
\begin{equation}\label{eq:D-U-choice}
        D_k=d_k+2,
        \qquad U_k=V_{k-1}^{\max}+1,
\end{equation}
and let
\begin{equation}\label{eq:Fk-def}
        F_k(x)=\sqrt{\frac{\lambda_k}{L_k}}
        \sum_{q=1}^{L_k}\phi_{d_k}(2^{U_k+qD_k}x).
\end{equation}
This is a block in the sense of \cref{sec:local}, with parameters
$(\lambda_k,B_k,L_k,d_k,D_k,U_k)$.  By \cref{lem:spike-basic}(c), the nonzero
Fourier coefficients of $F_k$ lie in the union of valuation bands
\begin{equation}\label{eq:Vk-def}
        \mathcal V_k=\bigcup_{q=1}^{L_k}
        [U_k+qD_k,\,U_k+qD_k+d_k-1].
\end{equation}
Put
\begin{equation}\label{eq:Vmax-update}
        V_k^{\max}=U_k+L_kD_k+d_k-1,
        \qquad
        Q_k=2^{U_k+L_kD_k+d_k+2},
\end{equation}
so that $Q_k$ is the threshold of \cref{lem:block-basic} for $F_k$.  Since
$U_k>V_{k-1}^{\max}$, the sets $\mathcal V_k$ are pairwise disjoint.

\medskip\noindent\emph{Placing the trials.}
Choose $M_{k,1}$ so large that
\begin{equation}\label{eq:first-start-conditions}
        M_{k,1}+D_k>m^{\rm last}_{k-1},
        \qquad
        U_k+M_{k,1}+D_k+1>\Omega_{k-1}.
\end{equation}
The first condition puts the new exponents after all previous ones, and the
second puts the digits used by the new good events after all previously used
ones.  For $1\le t<T_k$ define
\begin{equation}\label{eq:start-recursion}
        M_{k,t+1}=M_{k,t}+(L_k+\ell_{k,t})D_k+d_k+2,
\end{equation}
which moves past all digits used by trial $t$.  The $t$th trial contributes the
exponents
\begin{equation}\label{eq:Ikt-def}
        \mathcal I_{k,t}=\{M_{k,t}+D_k,\,M_{k,t}+2D_k,\ldots,
        M_{k,t}+\ell_{k,t}D_k\},
\end{equation}
which have the same spacing as the spikes in $F_k$.  The selected exponents are
strictly increasing, within each stage and from one stage to the next.  At the
end of stage $k$ put
\begin{equation}\label{eq:stage-updates}
        m_k^{\rm last}=M_{k,T_k}+\ell_{k,T_k}D_k,
        \qquad
        P_{k+1,1}=N_k^*,
\end{equation}
and
\begin{equation}\label{eq:Omega-update}
        \Omega_k=U_k+M_{k,T_k}+(L_k+\ell_{k,T_k})D_k+d_k.
\end{equation}
This is a convenient upper bound for the digits used by the good events of
stage $k$.

At the endpoint $N_{k,t}$ of trial $t$, the $\ell_{k,t}$ new exponents dominate
the $P_{k,t}$ earlier ones; this is what lets a successful trial control the
whole average.  The price is that the number of exponents grows geometrically in
the number of trials.

\begin{lemma}[Length recursion]\label{lem:length-recursion}
For every stage $k$ and every $1\le t\le T_k$,
\begin{equation}\label{eq:P-ell-ratio}
        P_{k,t}\le \eta\ell_{k,t},
        \qquad
        N_{k,t}\coloneqq P_{k,t}+\ell_{k,t}\le(1+\eta)\ell_{k,t}.
\end{equation}
Moreover there is a constant $C_\eta>1$, depending only on $\eta$, such that
\begin{equation}\label{eq:Nstar-growth-general}
        1+N_k^*\le (1+P_{k,1})C_\eta^{T_k}.
\end{equation}
\end{lemma}

\begin{proof}
The first inequality follows immediately from
$\ell_{k,t}\ge\eta^{-1}(P_{k,t}+1)$.  The second follows by adding
$\ell_{k,t}$.  Also
\[
        1+P_{k,t+1}=1+P_{k,t}+\ell_{k,t}
        \le C_\eta(1+P_{k,t})
\]
for a constant $C_\eta$ depending only on $\eta$.  Iterating gives
\eqref{eq:Nstar-growth-general}.
\end{proof}

For each trial let $\mathcal E_{k,t}$ denote the local good event from
\cref{def:local-good}, formed with the block $F_k$ and the trial
$(M_{k,t},\ell_{k,t})$.

\begin{lemma}[Independence and success probability]\label{lem:trial-prob-master}
The events $\mathcal E_{k,t}$, over all pairs $(k,t)$, are independent.  Moreover
\begin{equation}\label{eq:trial-prob-master}
        \Pr(\mathcal E_{k,t})\ge c_0\frac{\lambda_k}{B_k^2},
\end{equation}
where $c_0$ is the constant from \cref{lem:local-amplification}.  Consequently,
for the stage event $S_k=\bigcup_{t=1}^{T_k}\mathcal E_{k,t}$,
\begin{equation}\label{eq:stage-failure}
        \Pr(S_k^c)
        \le \exp\left(-c_0T_k\frac{\lambda_k}{B_k^2}\right).
\end{equation}
\end{lemma}

\begin{proof}
The event $\mathcal E_{k,t}$ depends only on digit windows
\[
        [U_k+M_{k,t}+hD_k+1,\,U_k+M_{k,t}+hD_k+d_k],
        \qquad \ell_{k,t}+1\le h\le L_k+1.
\]
The initial condition \eqref{eq:first-start-conditions} places the first such
window of stage $k$ beyond all windows used in earlier stages.  The recursion
\eqref{eq:start-recursion} places all windows of trial $t+1$ beyond all windows
of trial $t$.  Hence all good-trial events depend on disjoint binary digit
coordinates, and are independent.  The lower bound \eqref{eq:trial-prob-master}
is exactly \cref{lem:local-amplification}.  The estimate
\eqref{eq:stage-failure} follows from independence and $1-u\le e^{-u}$.
\end{proof}

\begin{proposition}[Main estimate]\label{prop:master}
Assume that stages are constructed as above and that
\begin{equation}\label{eq:lambda-summable-master}
        \sum_{k=1}^{\infty}\lambda_k<\infty.
\end{equation}
Let $(m_j)$ be the increasing enumeration of all selected exponents
$\bigcup_{k,t}\mathcal I_{k,t}$, and set $n_j=2^{m_j}$.  Let
\begin{equation}\label{eq:f-def-master}
        f=\sum_{k=1}^{\infty}F_k.
\end{equation}
Then $f$ converges in $L^2(\T)$ and almost everywhere to a real-valued mean-zero
function in $L^2(\T)$, and $(n_j)$ is lacunary with $n_{j+1}/n_j\ge2$.
Moreover, if for some fixed $2\le p<\infty$,
\begin{equation}\label{eq:lp-summable-master}
        \sum_{k=1}^{\infty}\lambda_kB_k^{p-2}<\infty,
\end{equation}
then the same series converges in $L^p(\T)$ and $f\in L^p(\T)$.

Let
\begin{equation}\label{eq:mu-def}
        \mu=\frac14\sum_{k=1}^{\infty}\frac{\lambda_k}{B_k},
\end{equation}
which is finite since $B_k\ge B_0$.  Then for almost every $x$, every successful
trial satisfies
\begin{equation}\label{eq:master-signal}
        \frac{1}{N_{k,t}}
        \sum_{j\le N_{k,t}} f(n_jx)
        \ge B_k-\mu.
\end{equation}
Consequently:
\begin{enumerate}[label=(\roman*)]
\item if $\sum_k \Pr(S_k)=\infty$, then almost surely there are infinitely many
stages $k$ for which some successful trial $t$ satisfies
\eqref{eq:master-signal};
\item if $\sum_k\Pr(S_k^c)<\infty$, then almost surely, for every sufficiently
large stage $k$, some successful trial $t$ satisfies \eqref{eq:master-signal}.
\end{enumerate}
\end{proposition}

\begin{proof}
The valuation sets $\mathcal V_k$ are disjoint, so the blocks $F_k$ are
orthogonal, and $\norm{F_k}_2^2=\lambda_k$ by \cref{lem:block-basic}.  Thus
\eqref{eq:lambda-summable-master} gives $L^2$ convergence to a real-valued
mean-zero function.  The blocks depend on disjoint finite sets of binary
digits, so they are independent mean-zero random variables with
$\sum_k\operatorname{Var}(F_k)=\sum_k\lambda_k<\infty$, and Kolmogorov's
convergence criterion gives almost-everywhere convergence of $\sum_k F_k$.

If \eqref{eq:lp-summable-master} holds for some $2\le p<\infty$, then
\cref{lem:block-lp} gives $\sum_k\norm{F_k}_p^p<\infty$.  Together with
$\sum_k\lambda_k<\infty$, \cref{lem:lp-series} shows convergence in
$L^p(\T)$.

The exponent intervals were placed in strictly increasing order, so the
increasing enumeration $(m_j)$ satisfies $m_{j+1}\ge m_j+1$.  Hence
$n_{j+1}/n_j\ge2$.

By \eqref{eq:block-lower-floor}, each block satisfies
$F_k\ge -\lambda_k/(4B_k)$.  Hence every finite partial sum satisfies
\[
        \sum_{i\le K}F_i(y)
        \ge -\frac14\sum_{i\le K}\frac{\lambda_i}{B_i}
        \ge -\mu.
\]
On the full-measure set where $\sum_iF_i(y)$ converges, passing to the limit
gives $f(y)\ge-\mu$.  Applying the same argument to each omitted series
$\sum_{i\ne k}F_i(y)$ and then intersecting over the countably many values of
$k$, we also have $f(y)-F_k(y)\ge-\mu$ for every $k$ on a common full-measure
set.  Pulling this set back under the countably many dyadic maps
$x\mapsto2^{m_j}x$, we may use these pointwise lower bounds at every selected
dilation for almost every $x$.

Fix such an $x$, and suppose $\mathcal E_{k,t}$ occurs.  At the endpoint
$N_{k,t}=P_{k,t}+\ell_{k,t}$,
\begin{align*}
\sum_{j\le N_{k,t}} f(2^{m_j}x)
&= \sum_{r=1}^{\ell_{k,t}}F_k(2^{M_{k,t}+rD_k}x)  \\
&\quad +\sum_{m_j<M_{k,t}+D_k} f(2^{m_j}x) \\
&\quad +\sum_{r=1}^{\ell_{k,t}}
      \bigl(f-F_k\bigr)(2^{M_{k,t}+rD_k}x).
\end{align*}
The first term is at least $2B_k\ell_{k,t}$ by
\cref{lem:local-amplification}.  The second has $P_{k,t}$ terms and is bounded
below by $-\mu P_{k,t}$; the third has $\ell_{k,t}$ terms and is bounded below
by $-\mu\ell_{k,t}$.  Since $P_{k,t}+\ell_{k,t}=N_{k,t}$,
\[
        \sum_{j\le N_{k,t}} f(2^{m_j}x)
        \ge 2B_k\ell_{k,t}-\mu N_{k,t}.
\]
Dividing by $N_{k,t}$ and using \cref{lem:length-recursion} gives
\[
        \frac{1}{N_{k,t}}\sum_{j\le N_{k,t}} f(2^{m_j}x)
        \ge 2B_k\frac{\ell_{k,t}}{N_{k,t}}-\mu
        \ge \frac{2B_k}{1+\eta}-\mu.
\]
Since $\eta=1/20$, the last quantity is at least $B_k-\mu$, which proves
\eqref{eq:master-signal}.

The stage events $S_k$ are independent by \cref{lem:trial-prob-master}.  If
$S_k$ occurs, then at least one trial $\mathcal E_{k,t}$ occurs, and the estimate just
proved applies to that trial.  Part (i) follows from the second
Borel--Cantelli lemma, and part (ii) follows from the first Borel--Cantelli
lemma.
\end{proof}

\section{Fourier tails and the endpoint construction}\label{sec:endpoint}

We now choose the parameters to prove \cref{thm:endpoint}.  As explained in
Section~\ref{subsec:roadmap}, the key choice is to run $T_k=\lceil1/\lambda_k\rceil$
trials at stage $k$.  The next lemma shows that the Fourier threshold $Q_k$ is
then at most double-exponential in $1/\lambda_k$.

\begin{lemma}[Size of the Fourier threshold]\label{lem:endpoint-scale}
There is a constant $C$, depending only on $\eta$, with the following property.
Suppose that in stage $k$ of the construction of \cref{sec:master},
\begin{equation}\label{eq:T-endpoint}
        T_k=\left\lceil \frac{1}{\lambda_k}\right\rceil,
\end{equation}
and that $\lambda_k$ is so small that
\begin{equation}\label{eq:T-dominates-past}
        T_k\ge \log(2+P_{k,1}+V_{k-1}^{\max}+B_k).
\end{equation}
Then
\begin{equation}\label{eq:loglogQ-asymp}
        \log\log Q_k\le \frac{C}{\lambda_k}.
\end{equation}
\end{lemma}

\begin{proof}
Throughout the proof $C$ denotes a constant depending only on $\eta$, which may
change from line to line.  Write
\[
        Q_k=2^{E_k},
        \qquad E_k=U_k+L_kD_k+d_k+2.
\]
By \cref{lem:length-recursion},
\[
        L_k\le C(1+P_{k,1})C_\eta^{T_k}.
\]
By \eqref{eq:T-dominates-past},
\[
        1+P_{k,1}\le e^{T_k},
        \qquad B_k\le e^{T_k},
        \qquad U_k=V_{k-1}^{\max}+1\le e^{T_k},
\]
and hence $L_k\le e^{CT_k}$.  From the choice of $d_k$,
\[
        2^{d_k}<128\frac{B_k^2L_k}{\lambda_k},
\]
so
\[
        d_k\le C\bigl(1+\log B_k+\log L_k+\log(1/\lambda_k)\bigr)\le CT_k.
\]
Here we used $\log(1/\lambda_k)\le 1/\lambda_k\le T_k$.  Thus
$D_k=d_k+2\le CT_k$, and
\[
        E_k=U_k+L_kD_k+d_k+2
        \le e^{T_k}+e^{CT_k}CT_k+CT_k+2
        \le e^{CT_k}.
\]
Consequently
\[
        \log\log Q_k=\log(E_k\log2)\le CT_k
        \le \frac{2C}{\lambda_k},
\]
since $T_k\le1/\lambda_k+1\le2/\lambda_k$.
\end{proof}

\begin{remark}[The exponent $1/2$]\label{rem:endpoint-limit}
By Matsuyama's theorem, \cref{thm:endpoint} cannot hold with
$(\log\log N)^{-c}$ for any $c>1/2$.  It is instructive to see where the
construction meets this barrier.  Let $\tau_k=T_1+\dots+T_k$ be the total number
of trials in stages $1,\dots,k$.  Since $P_{k,t+1}+1\ge21(P_{k,t}+1)$ and
$P_{1,1}=0$, we have $P_{k,T_k}+1\ge21^{\tau_k-1}$, hence
$L_k\ge160\cdot21^{\tau_k-1}$ and $\log\log Q_k\ge\tau_k$, for every choice of
parameters in \cref{sec:master}.  (In particular \cref{lem:endpoint-scale} is
sharp up to the constant.)  Next, every nonzero frequency in the valuation band
$[a,b]$ has absolute value at least $2^a$, so the spikes of $F_k$ with index
$q>L_k/2$ have all their Fourier mass above $Q_k^{1/2}$; thus
$\norm{(I-S_N)F_k}_2^2\ge\lambda_k/2$ for $N<Q_k^{1/2}$.  A tail bound
$\norm{f-S_Nf}_2\ll(\log\log N)^{-c}$ would therefore force
$\lambda_k\ll(\log\log Q_k)^{-2c}\le\tau_k^{-2c}$.  A single trial succeeds with
probability at most $\lambda_k/(64B_k^2)$, so the expected number of successful
trials in stage $k$ would be $\ll T_k\tau_k^{-2c}B_k^{-2}\le T_k\tau_k^{-2c}$.  For
$c>1/2$ these bounds are summable, since
$T_k\tau_k^{-2c}\le\int_{\tau_{k-1}}^{\tau_k}s^{-2c}\,ds$ for $k\ge2$, and the
first Borel--Cantelli lemma shows that almost surely only finitely many stages
would succeed.  For $c=1/2$ the bound is $(T_k/\tau_k)B_k^{-2}$.  Under
\eqref{eq:T-dominates-past} we have $T_k\ge\log(2+P_{k,1})\ge\tau_{k-1}$, so this
is of the same order as the lower bound $\Pr(S_k)\gg B_k^{-2}$ used in the proof
of \cref{thm:endpoint}, and $\sum_kB_k^{-2}=\infty$ is still possible.
\end{remark}

The next proposition sums the Fourier tails of all blocks.

\begin{proposition}[Endpoint Fourier-tail summation]\label{prop:endpoint-tail}
Suppose that the construction of \cref{sec:master} satisfies, for all
sufficiently large $k$,
\begin{equation}\label{eq:lambda-Q-regular}
        \lambda_k\ll\frac{1}{\log\log Q_k},
\end{equation}
\begin{equation}\label{eq:future-lambda-tail}
        \sum_{i>k}\lambda_i\ll\lambda_{k+1},
\end{equation}
and
\begin{equation}\label{eq:weighted-Q-separation}
        \lambda_kQ_k\ge 2^k\left(1+\sum_{i<k}\lambda_iQ_i\right).
\end{equation}
Assume also that $Q_k$ is eventually strictly increasing.  Then the function
$f=\sum_kF_k$ satisfies
\begin{equation}\label{eq:endpoint-tail-prop}
        \norm{f-S_Nf}_2^2\ll \frac{1}{\log\log N}
\end{equation}
for all sufficiently large $N$.
\end{proposition}

\begin{proof}
For each $k$, \cref{lem:block-basic} gives
\begin{equation}\label{eq:rho-k-tail}
        \rho_k(N)^2\coloneqq \norm{(I-S_N)F_k}_2^2
        \ll \lambda_k\min\left(1,\frac{Q_k}{N}\right).
\end{equation}
The valuation sets $\mathcal V_k$ are disjoint, and applying $I-S_N$ preserves
this disjointness.  Parseval therefore gives
\begin{equation}\label{eq:orth-tail-sum}
        \norm{f-S_Nf}_2^2=\sum_k\rho_k(N)^2.
\end{equation}

The finitely many initial blocks only contribute $O(1/N)$ to the squared tail,
which is $O((\log\log N)^{-1})$ for large $N$.  We therefore ignore them.
Since $Q_k$ is eventually strictly increasing, we may fix $N$ large and choose
$k$ with $Q_k\le N<Q_{k+1}$.  The past and current
blocks satisfy, by \eqref{eq:rho-k-tail} and \eqref{eq:weighted-Q-separation},
\[
        \sum_{i\le k}\rho_i(N)^2
        \ll \frac{1}{N}\sum_{i\le k}\lambda_iQ_i
        \ll \frac{\lambda_kQ_k}{N}.
\]
Since $N/\log\log N$ is increasing for large $N$ and $N\ge Q_k$,
\[
        \frac{Q_k}{N}\le \frac{\log\log Q_k}{\log\log N}.
\]
Together with \eqref{eq:lambda-Q-regular}, this yields
\[
        \sum_{i\le k}\rho_i(N)^2\ll \frac{1}{\log\log N}.
\]
For future blocks, \eqref{eq:rho-k-tail} and \eqref{eq:future-lambda-tail} give
\[
        \sum_{i>k}\rho_i(N)^2\le\sum_{i>k}\lambda_i
        \ll\lambda_{k+1}
        \ll \frac{1}{\log\log Q_{k+1}}
        \le \frac{1}{\log\log N},
\]
where the last inequality uses $N<Q_{k+1}$.  Combining the two estimates proves
\eqref{eq:endpoint-tail-prop}.
\end{proof}

\begin{proof}[Proof of \cref{thm:endpoint}]
Choose a nondecreasing sequence $B_k\ge B_0$ such that $B_k\to\infty$ and
\begin{equation}\label{eq:B-divergent}
        \sum_{k=1}^{\infty}\frac{1}{B_k^2}=\infty;
\end{equation}
for instance $B_k=B_0+\sqrt{\log(k+2)}$.  We construct the stages recursively.
At stage $k$, after the previous stages have been fixed, choose $\lambda_k>0$
so small that
\begin{align}
        &\lambda_k\le 2^{-k},                                            \label{eq:end-lambda-sum}\\
        &\lambda_k\le 2^{-k-2}\lambda_{k-1}\quad(k\ge2),                 \label{eq:end-lambda-tail-cond}\\
        &C_r\lambda_kB_k^{r-2}\le 2^{-k}
          \quad(2\le r\le k),                                             \label{eq:end-lp-diagonal}\\
        &T_k\coloneqq \left\lceil \frac{1}{\lambda_k}\right\rceil
          \ge \log(2+P_{k,1}+V_{k-1}^{\max}+B_k),                        \label{eq:end-T-dominates}\\
        &\lambda_kQ_k\ge 2^k\left(1+\sum_{i<k}\lambda_iQ_i\right),        \label{eq:end-weighted-sep}
\end{align}
where $Q_k$ is the threshold produced by the stage built with these parameters,
and $C_r$ is the constant from \cref{lem:block-lp}.
To see that such a choice is possible, temporarily build the candidate stage for
each small value of $\lambda$.  Then $T=\lceil1/\lambda\rceil\to\infty$,
and the length recursion gives $L(\lambda)\ge21^{T-1}$, as in
\cref{rem:endpoint-limit}.  Since
$Q(\lambda)\ge2^{L(\lambda)}$, we have
$\lambda Q(\lambda)\to\infty$ as $\lambda\downarrow0$.  All upper bound
conditions above hold for all sufficiently small $\lambda$, and the weighted
separation condition holds after making $\lambda$ smaller if necessary.

The sequence $(\lambda_k)$ is summable, so \cref{prop:master} gives a
mean-zero $f\in L^2(\T)$ and a dyadic lacunary sequence.  In fact, $f$ belongs
to every finite $L^r$ space.  Indeed, fix an integer $r\ge2$.  The finitely
many blocks with $k<r$ are bounded, and for $k\ge r$, \cref{lem:block-lp} and
\eqref{eq:end-lp-diagonal} give $\norm{F_k}_r^r\le2^{-k}$.  Since
$\sum_k\lambda_k<\infty$, \cref{lem:lp-series} shows convergence in
$L^r(\T)$.  Since $\T$ has measure one, $L^r(\T)\subset L^p(\T)$ for $p\le r$,
and hence $f\in\bigcap_{1\le p<\infty}L^p(\T)$.

By \cref{lem:endpoint-scale}, condition \eqref{eq:lambda-Q-regular} holds.  The
geometric decay \eqref{eq:end-lambda-tail-cond} gives
\eqref{eq:future-lambda-tail}, and \eqref{eq:end-weighted-sep} is exactly
\eqref{eq:weighted-Q-separation}.  Moreover, for $k\ge2$,
\eqref{eq:end-weighted-sep} gives
$\lambda_kQ_k\ge2^k\lambda_{k-1}Q_{k-1}$, while
\eqref{eq:end-lambda-tail-cond} gives
$\lambda_k\le2^{-k-2}\lambda_{k-1}$.  Hence
$Q_k\ge2^{2k+2}Q_{k-1}$, so the thresholds are strictly increasing from
stage $2$ onward.  Therefore \cref{prop:endpoint-tail} gives
\[
        \norm{f-S_Nf}_2\ll (\log\log N)^{-1/2}.
\]

It remains to prove the divergent limsup.  By \cref{lem:trial-prob-master},
\[
        \Pr(S_k)
        \ge 1-\exp\left(-c_0T_k\frac{\lambda_k}{B_k^2}\right)
        \ge \frac{c}{B_k^2},
\]
with $c>0$ depending only on $c_0$ and $B_0$.  The last inequality
uses $T_k\lambda_k\ge1$ and the elementary bound
$1-e^{-a/u}\ge c_a/u$ for $u\ge B_0^2$.  By \eqref{eq:B-divergent} and the
independence of the stage events, the second Borel--Cantelli lemma gives that
$S_k$ occurs for infinitely many $k$ almost surely.  Along each such stage,
choose a successful trial $t=t(k)$.  Then \cref{prop:master} gives
\[
        \frac{1}{N_{k,t}}\sum_{j\le N_{k,t}}f(n_jx)
        \ge B_k-\mu.
\]
The corresponding endpoints tend to infinity, because each stage inserts at
least one new exponent and $N_{k,t}\ge P_{k,1}=N_{k-1}^*$ for $k\ge2$.
Since $B_k\to\infty$, the limsup is $+\infty$ almost surely.
\end{proof}

\section{Large partial sums in finite \texorpdfstring{$L^p$}{Lp}}\label{sec:large-sums}

We now prove \cref{thm:large-partial-sums}.  The Fourier tail plays no role
here, so we can afford enough trials to make stage failure summable.  The
budget is now the $L^p$ norm, and we take the cost $\lambda_k$ to be a negative
power of the target height $B_k$.

\begin{proof}[Proof of \cref{thm:large-partial-sums}]
Fix $2\le p<\infty$, and let $a_k=2^{-k-2}$; this is the $L^p$ budget of
stage $k$.

We construct the stages recursively.  Suppose stages $1,\ldots,k-1$ have already
been fixed.  We leave the target height $B_k\ge B_0$ free for the moment, and
for each proposed value of $B_k$ we set
\begin{equation}\label{eq:lambda-large-p}
        \lambda_k=a_kB_k^{-(p-2)},
\end{equation}
so that $\lambda_kB_k^{p-2}=a_k$ and $0<\lambda_k\le a_k\le1$.

Next choose the number of trials by
\begin{equation}\label{eq:T-large-sums}
        T_k=\left\lceil \Gamma \frac{B_k^2}{\lambda_k}\log(k+1)\right\rceil
        =\left\lceil \Gamma \frac{B_k^p}{a_k}\log(k+1)\right\rceil,
\end{equation}
where $\Gamma\ge 8/c_0$ is fixed.  Then $T_k\lambda_k/B_k^2\ge\Gamma\log(k+1)$,
which will make the probability of stage failure summable in $k$.

We first bound the stage endpoint $N_k^*$ in terms of $B_k$.  By
\cref{lem:length-recursion},
\[
        1+N_k^*\le (1+P_{k,1})C_\eta^{T_k}.
\]
Here $P_{k,1}$ is already fixed when stage $k$ begins.  Hence
\begin{equation}\label{eq:large-logN-bound}
        \log N_k^*
        \le C_k+C T_k
        \le C'_k\frac{B_k^p}{a_k}\log(k+1),
\end{equation}
where the constants may depend on the earlier stages, on $\eta$, and on the
fixed choices of $p$ and $\Gamma$, but not on the proposed value of $B_k$.

Let
\[
        \varepsilon_k=\frac{1}{2p(k+2)}.
\]
Then $1/p-\varepsilon_k>0$.  From \eqref{eq:large-logN-bound},
\[
        (\log N_k^*)^{1/p-\varepsilon_k}
        \ll_k
        B_k^{1-p\varepsilon_k}
        a_k^{-1/p+\varepsilon_k}
        (\log(k+1))^{1/p-\varepsilon_k}.
\]
Therefore
\[
        \frac{B_k}{(\log N_k^*)^{1/p-\varepsilon_k}}
        \gg_k
        B_k^{p\varepsilon_k}
        a_k^{1/p-\varepsilon_k}
        (\log(k+1))^{-1/p+\varepsilon_k}.
\]
For the fixed stage $k$, all factors except $B_k^{p\varepsilon_k}$ are fixed,
and $p\varepsilon_k>0$.  Hence
\[
        \frac{B_k}{(\log N_k^*)^{1/p-\varepsilon_k}}
        \longrightarrow\infty
        \qquad\text{as } B_k\to\infty.
\]

We must also choose $B_k$ large enough to dominate the constant $\mu$ from
\cref{prop:master}, which depends on all stages, including future ones.  For a
proposed value of $B_k$, define
\begin{equation}\label{eq:mu-k-plus}
        \overline\mu_k(B_k)
        =\frac14\sum_{i<k}\frac{\lambda_i}{B_i}
        +\frac14\,\frac{a_k}{B_k^{p-1}}
        +\frac14\sum_{i>k}\frac{a_i}{B_0^{p-1}}.
\end{equation}
The first sum is the contribution of stages already chosen.  The middle term is
the contribution of the current stage, since
\[
        \frac{\lambda_k}{B_k}
        =
        \frac{a_k}{B_k^{p-1}}.
\]
The final sum bounds all future contributions, because every later stage will
satisfy $B_i\ge B_0$ and hence
\[
        \frac{\lambda_i}{B_i}
        =
        \frac{a_i}{B_i^{p-1}}
        \le
        \frac{a_i}{B_0^{p-1}}.
\]
Thus, after all future stages have been chosen, the constant $\mu$ in
\cref{prop:master} will satisfy
\[
        \mu\le \overline\mu_k(B_k).
\]
For fixed $k$, the quantity $\overline\mu_k(B_k)$ is bounded as
$B_k\to\infty$, since only the middle term depends on $B_k$.  Combining this
with the divergence above, we may choose $B_k$ so large that
\begin{equation}\label{eq:B-large-choice}
        \frac{B_k-\overline\mu_k(B_k)}
        {(\log N_k^*)^{1/p-\varepsilon_k}}
        \ge k.
\end{equation}
We then freeze this value of $B_k$ and complete stage $k$ with the parameters
defined by \eqref{eq:lambda-large-p} and \eqref{eq:T-large-sums}.  This completes
the recursive construction of all stages.

The resulting function belongs to $L^p(\T)$.  Indeed,
\[
        \sum_k\lambda_k
        \le
        \sum_k a_k
        <\infty,
\]
so the $L^2$ summability hypothesis of \cref{prop:master} holds.  Moreover,
\cref{lem:block-lp} gives
\[
        \norm{F_k}_p^p
        \le C_p\lambda_kB_k^{p-2}
        = C_pa_k,
\]
and therefore $\sum_k\norm{F_k}_p^p<\infty$.  Hence
\cref{lem:lp-series} gives convergence in $L^p(\T)$, and \cref{prop:master}
applies to the constructed function and lacunary sequence.

It remains to show that the successful stages occur eventually almost surely.
By \cref{lem:trial-prob-master} and \eqref{eq:T-large-sums},
\[
        \Pr(S_k^c)
        \le \exp\left(-c_0T_k\frac{\lambda_k}{B_k^2}\right)
        \le \exp\left(-c_0\Gamma\log(k+1)\right)
        \le (k+1)^{-8}.
\]
The failure probabilities are summable.  Therefore \cref{prop:master} implies
that, for almost every $x$, every sufficiently large stage $k$ has some trial
endpoint $N_{k,t}\le N_k^*$ such that
\[
        \frac{1}{N_{k,t}}
        \sum_{j\le N_{k,t}}f(n_jx)
        \ge B_k-\mu.
\]

Fix such an $x$ and such a large $k$.  Since $N_{k,t}\le N_k^*$ and
$1/p-\varepsilon_k>0$,
\[
        (\log N_{k,t})^{1/p-\varepsilon_k}
        \le
        (\log N_k^*)^{1/p-\varepsilon_k}.
\]
Also $\mu\le\overline\mu_k(B_k)$ by the construction of
$\overline\mu_k(B_k)$.  Hence \eqref{eq:B-large-choice} gives
\[
        \frac{\sum_{j\le N_{k,t}}f(n_jx)}
        {N_{k,t}(\log N_{k,t})^{1/p-\varepsilon_k}}
        \ge
        \frac{B_k-\mu}{(\log N_k^*)^{1/p-\varepsilon_k}}
        \ge
        \frac{B_k-\overline\mu_k(B_k)}
        {(\log N_k^*)^{1/p-\varepsilon_k}}
        \ge k.
\]
Thus, along one trial endpoint from every sufficiently large stage, the
normalized partial sums with exponent $1/p-\varepsilon_k$ are at least $k$.

Now fix any $\varepsilon>0$.  Since $\varepsilon_k\to0$, we have
$\varepsilon_k<\varepsilon$ for all sufficiently large $k$.  For such $k$,
\[
        1/p-\varepsilon
        \le
        1/p-\varepsilon_k,
\]
so replacing $\varepsilon_k$ by $\varepsilon$ only decreases the logarithmic
denominator.  Therefore the same endpoints satisfy
\[
        \frac{\sum_{j\le N_{k,t}}f(n_jx)}
        {N_{k,t}(\log N_{k,t})^{1/p-\varepsilon}}
        \ge k
\]
for all sufficiently large $k$.  The endpoints tend to infinity, since each
stage appends at least one new exponent.  This proves
\eqref{eq:large-partial-sums}.

Finally take $p=2$.  For any fixed $0<\varepsilon<1/2$,
\[
        \frac{(\log N)^{1/2-\varepsilon}}{\sqrt{\log\log N}}
        \longrightarrow\infty.
\]
Thus, if the partial sums were $o(N\sqrt{\log\log N})$ almost everywhere, then
the normalized quantities in \eqref{eq:large-partial-sums} would tend to $0$
for this choice of $\varepsilon$, contradicting the infinite limsup.  Hence the
bound $o(N\sqrt{\log\log N})$ asked about by Erd\H{o}s does not hold in general.
\end{proof}

\begin{remark}[A matching upper bound]\label{rem:lp-upper}
For any increasing sequence $(n_j)$ of positive integers and any
$f\in L^p(\T)$, $1\le p<\infty$, one has the almost-everywhere upper bound
\[
        \sum_{j\le N}f(n_jx)=O_{f,p,\varepsilon}
        \bigl(N(\log N)^{1/p+\varepsilon}\bigr)
        \qquad (\varepsilon>0).
\]
Indeed, for $2^m\le N<2^{m+1}$,
\[
        \max_{2^m\le N<2^{m+1}}\left|\sum_{j\le N}f(n_jx)\right|
        \le \sum_{j<2^{m+1}}|f(n_jx)|.
\]
Each map $x\mapsto n_jx$ preserves Lebesgue measure, so the $L^p$ norm of the
right-hand side is at most $2^{m+1}\norm{f}_p$.  Chebyshev's inequality gives
exceptional sets of measure $O(m^{-1-p\varepsilon})$ after normalizing by
$2^m m^{1/p+\varepsilon}$, and Borel--Cantelli completes the argument.  Thus
\cref{thm:large-partial-sums} is sharp in the logarithmic exponent up to
$\varepsilon$.
\end{remark}

\section{A bounded dyadic hitting-set construction}\label{sec:bounded}

We prove \cref{thm:bounded}.  The proof does not use the spike blocks, but it
uses the same arrangement of stages and trials: at stage $k$ we build a small set
$E_k$, and a central hit forces every point of a trial into $E_k$.  The
structure follows Marstrand's proof \cite[Lemmas~3.4, 3.5 and
Theorem~3.2]{Marstrand1970}.  The hitting event $\mathcal H_{k,t}$ below plays
the role of his set $G^*$.  The independent trials within a stage play the role
of his independent dilated copies.  The parameter $\theta_k\to0$ plays the role
of his $q\to\infty$.

\begin{proof}[Proof of \cref{thm:bounded}]
Fix $0<\varepsilon<1$.  Choose integers $A_k\ge4$ such that
\begin{equation}\label{eq:bounded-A-sum}
        \sum_{k=1}^{\infty}A_k^{-1}<\varepsilon.
\end{equation}
Let $\theta_k=(k+1)^{-1}$.  Put $c_1=7/32$, and choose $C\ge8/c_1$.

Let $P_{k,1}$ be the number of selected exponents before stage $k$, with
$P_{1,1}=0$.  At stage $k$ set
\begin{equation}\label{eq:bounded-T}
        T_k=\lceil CA_k\log(k+1)\rceil.
\end{equation}
Starting from $P_{k,1}$, define
\begin{equation}\label{eq:bounded-lengths}
        \ell_{k,t}=\lceil\theta_k^{-1}(P_{k,t}+1)\rceil,
        \qquad P_{k,t+1}=P_{k,t}+\ell_{k,t}.
\end{equation}
Then $P_{k,t}\le\theta_k\ell_{k,t}$.  Put
\[
        L_k=8\ell_{k,T_k},
\]
choose $d_k$ so that
\[
        A_kL_k\le2^{d_k}<2A_kL_k,
\]
and set $D_k=d_k+2$.  Define
\begin{equation}\label{eq:bounded-Ek}
        E_k=\bigcup_{q=1}^{L_k}
        \{y\in\T:\{2^{qD_k}y\}<2^{-d_k}\}.
\end{equation}
Then $|E_k|\le L_k2^{-d_k}\le A_k^{-1}$.  Let
\[
        E=\bigcup_{k=1}^{\infty}E_k.
\]
By \eqref{eq:bounded-A-sum}, $|E|<\varepsilon$.

We now choose the exponents.  Suppose stages before $k$ have been selected and
let $m_{k-1}^{\rm last}$ be the largest previously selected exponent, with $m_0^{\rm last}=0$.
Take
\[
        M_{k,1}=m_{k-1}^{\rm last}+D_k+1,
\]
and, for $1\le t<T_k$,
\[
        M_{k,t+1}=M_{k,t}+(L_k+2)D_k+d_k+2.
\]
The $t$th trial uses exactly the $\ell_{k,t}$ exponents
\[
        \mathcal I_{k,t}=\{M_{k,t}+rD_k:1\le r\le\ell_{k,t}\}.
\]
These exponent intervals are strictly increasing.  After completing stage $k$,
set
\[
        m_k^{\rm last}=M_{k,T_k}+\ell_{k,T_k}D_k,
        \qquad
        P_{k+1,1}=P_{k,T_k+1}.
\]
Continuing inductively over all stages, let $(m_j)$ be the increasing
enumeration of the union of all intervals $\mathcal I_{k,t}$ and put $n_j=2^{m_j}$.

For a fixed trial define
\[
        \mathcal H_{k,t}=\bigcup_{h=\ell_{k,t}+1}^{L_k+1}
        \{x\in\T:\{2^{M_{k,t}+hD_k}x\}<2^{-d_k}\}.
\]
If $\mathcal H_{k,t}$ occurs, choose an $h$ witnessing this event.  For each
$1\le r\le\ell_{k,t}$, put $q=h-r$.  Then $1\le q\le L_k$ and
\[
        2^{qD_k}(2^{M_{k,t}+rD_k}x)=2^{M_{k,t}+hD_k}x,
\]
so every point of the trial lands in $E_k\subset E$.

The events in the union defining $\mathcal H_{k,t}$ have disjoint digit windows.  With
$p_k=2^{-d_k}$ and $n_{k,t}=L_k-\ell_{k,t}+1\ge7L_k/8$, we have
$p_k>1/(2A_kL_k)$ and $n_{k,t}p_k\le1/A_k\le1/2$.  Hence
\[
        \Pr(\mathcal H_{k,t})=1-(1-p_k)^{n_{k,t}}
        \ge \frac{1}{2}n_{k,t}p_k
        \ge \frac{7}{32A_k}=\frac{c_1}{A_k}.
\]
The recursion for the starts makes the events $\mathcal H_{k,t}$ independent within a
fixed stage.  Therefore
\[
        \Pr(\text{no good trial at stage }k)
        \le \left(1-\frac{c_1}{A_k}\right)^{T_k}
        \le (k+1)^{-8}.
\]
Borel--Cantelli gives that, for almost every $x$, every sufficiently large
stage has a good trial.  For such a trial,
\[
        \sum_{j\le N_{k,t}}\1_E(n_jx)\ge\ell_{k,t},
        \qquad
        N_{k,t}=P_{k,t}+\ell_{k,t}\le(1+\theta_k)\ell_{k,t}.
\]
Thus
\[
        \frac{1}{N_{k,t}}\sum_{j\le N_{k,t}}\1_E(n_jx)
        \ge \frac{1}{1+\theta_k}.
\]
Letting $k\to\infty$ along successful stages gives limsup at least $1$, and the
reverse inequality is trivial.  Hence the limsup is exactly $1$ almost
everywhere.

Finally, put $g=\1_E-|E|$.  If the averages
$N^{-1}\sum_{j\le N}g(n_jx)$ converged almost everywhere, their limit would be
bounded and would have integral $0$ by dominated convergence, since each dyadic
map preserves Lebesgue measure and hence each average has integral $0$.  But
for almost every $x$ the preceding paragraph gives
\[
        \limsup_{N\to\infty}\frac1N\sum_{j\le N}g(n_jx)=1-|E|>0.
\]
A convergent sequence with this limsup would have positive limit almost
everywhere, contradicting integral $0$.
\end{proof}

\section{Further questions}\label{sec:questions}

\Cref{thm:endpoint} and Matsuyama's theorem leave open a finer boundary
problem.  Write $u=\log\log N$ and consider moduli
\[
        \omega(N)=u^{-1/2}\psi(u),
\]
where $\psi$ is slowly varying.  A natural guess is that the boundary between
convergence and divergence is governed by square summability on the log-log
scale,
\[
        \sum_r \omega(e^{e^r})^2<\infty,
        \quad\text{that is,}\quad
        \sum_r \frac{\psi(r)^2}{r}<\infty.
\]
For $\psi(r)=(\log r)^\beta$, this condition changes at $\beta=-1/2$.  The present
paper reaches $\psi\equiv1$ on the divergence side, but does not attempt to
identify the optimal slowly varying boundary.

The sequences constructed here are very irregular: they consist of long
well-separated blocks of exponents with rapidly growing gaps between them.  It
is natural to ask what happens under structural restrictions such as two-sided
lacunarity
\[
        1+\delta\le \frac{n_{j+1}}{n_j}\le \Lambda,
\]
or, in the dyadic model, bounded exponent gaps
$1\le m_{j+1}-m_j\le H$.  This lies between the present construction and the
pure geometric case covered by Raikov's theorem.

For $f\in L^p$ with $p<\infty$, the almost-sure size of lacunary partial sums
is determined up to $\varepsilon$ in the exponent by \cref{thm:large-partial-sums}
and \cref{rem:lp-upper}.  It remains to determine the correct order without
$\varepsilon$ losses, for instance up to secondary logarithmic factors.  The case
$p=\infty$ is qualitatively different, since bounded functions have the trivial
$O(N)$ upper bound and the bounded construction in \cref{thm:bounded} is a
sweeping-out rather than a logarithmic-growth phenomenon.

Finally, bounded counterexamples of the kind in \cref{thm:bounded} cannot be
Riemann integrable.  For each fixed increasing integer sequence
$(n_j)$, Weyl's theorem \cite{Weyl1916} implies that $(n_jx)$ is uniformly
distributed modulo one for almost every $x$.  Hence every Riemann integrable $g$ satisfies
$N^{-1}\sum_{j\le N}g(n_jx)\to\int_\T g$ almost everywhere.  It would be
interesting to locate a sharper regularity boundary between this classical
positive behavior and measurable sweeping-out examples.

\section*{Declarations}

\noindent\textbf{Funding.}
No funding was received to assist with the preparation of this manuscript.

\medskip

\noindent\textbf{Competing interests.}
The author declares that he has no relevant financial or non-financial interests to disclose.

\medskip

\noindent\textbf{Data availability.}
Data sharing is not applicable to this article as no datasets were generated or analyzed.

\medskip

\noindent\textbf{Use of AI-assisted technologies.}
The author used GPT-5.4 Pro during the development of this work to explore proof
strategies, test intermediate formulations, and assist with exposition. During the
revision, the author used Claude Opus 5.5 to assist with editing the exposition and
checking references. All mathematical
arguments and claims in the final manuscript were independently verified by the
author, who takes full responsibility for the paper.

\medskip

\noindent\textbf{Author contribution.}
The sole author is responsible for all aspects of the work.

\section*{Acknowledgements}

The author thanks Alyxia Seah for thoughtful comments on an earlier draft; in particular,
her suggestions led to the current improved definition of a good trial event.  The author
is grateful to the anonymous referee for a careful reading and helpful suggestions, in
particular for pointing out the connection with Marstrand's work.

\end{document}